\documentclass[11pt]{article}
\usepackage{tikz}

\usepackage{empheq}
\usepackage[shortlabels]{enumitem}
\setlist[enumerate]{nosep}
\usepackage[doc]{optional}
\usepackage{xcolor}
\usepackage[colorlinks=true,
            linkcolor=refkey,
            urlcolor=lblue,
            citecolor=red]{hyperref}
\usepackage{float}
\usepackage{soul}
\usepackage{graphicx}
\definecolor{labelkey}{rgb}{0,0.08,0.45}
\definecolor{refkey}{rgb}{0,0.6,0.0}
\definecolor{Brown}{rgb}{0.45,0.0,0.05}
\definecolor{lime}{rgb}{0.00,0.8,0.0}
\definecolor{lblue}{rgb}{0.5,0.5,0.99}

 \usepackage{mathpazo}

\usepackage{stmaryrd}
\usepackage{amssymb}
\newcommand{\seppone}{\setlength{\itemsep}{-1pt}}

\newcommand{\nnn}{\ensuremath{{n\in{\mathbb N}}}}

\newcommand{\menge}[2]{\big\{{#1}~\big |~{#2}\big\}}

\newcommand{\To}{\ensuremath{\rightrightarrows}}

\newcommand{\fenv}[1]%
{\ensuremath{\,\overrightarrow{\operatorname{env}}_{#1}}}
\newcommand{\benv}[1]%
{\ensuremath{\,\overleftarrow{\operatorname{env}}_{#1}}}

\newcommand{\scal}[2]{\left\langle{#1},{#2}  \right\rangle}

\newcommand{\RR}{\ensuremath{\mathbb R}}
\newcommand{\RP}{\ensuremath{\mathbb{R}_+}}
\newcommand{\RPP}{\ensuremath{\mathbb{R}_{++}}}
\newcommand{\RM}{\ensuremath{\mathbb{R}_-}}

\newcommand{\dom}{\ensuremath{\operatorname{dom}}}

\newcommand{\ran}{\ensuremath{\operatorname{ran}}}
\newcommand{\zer}{\ensuremath{\operatorname{zer}}}

\newcommand{\Fix}{\ensuremath{\operatorname{Fix}}}
\newcommand{\Id}{\ensuremath{\operatorname{Id}}}

{\begin{list}{}{%
\settowidth{\labelwidth}{\textrm{#1~}}%
\setlength{\leftmargin}{\labelwidth+\labelsep}}}
{\end{list}}
\usepackage{amsthm}
\newtheorem{theorem}{Theorem}[section]

\newtheorem{lem}[theorem]{Lemma}

\newtheorem{proposition}[theorem]{Proposition}
\newtheorem{prop}[theorem]{Proposition}

\newtheorem{ex}[theorem]{Example}

\newtheorem{remark}[theorem]{Remark}

\providecommand{\norm}[1]{\lVert#1\rVert}

\providecommand{\normsq}[1]{\lVert#1\rVert^2}

\providecommand{\innp}[1]{\langle#1\rangle}
\providecommand{\iNNp}[1]{\Big\langle#1\Big\rangle}

\providecommand{\RR}{\mathbb{R}}

\providecommand{\ran}{\operatorname{ran}}

\providecommand{\dom}{\operatorname{dom}}

\providecommand{\gra}{\operatorname{gra}}
\providecommand{\Id}{\operatorname{{ Id}}}

\providecommand{\To}{\rightrightarrows}

\providecommand{\ran}{\operatorname{ran}}

\providecommand{\Id}{\operatorname{Id}}

\providecommand{\spn}{\operatorname{span}}
\providecommand{\sinc}{\operatorname{sinc}}
\providecommand{\zer}{\operatorname{zer}}

\providecommand{\RR}{\mathbb{R}}

\definecolor{myblue}{rgb}{.8, .8, 1}
  \newcommand*\mybluebox[1]{%
    \colorbox{myblue}{\hspace{1em}#1\hspace{1em}}}

\allowdisplaybreaks 

\begin{document}

\title{ \textsc
Intriguing maximally monotone operators derived \\ from 
nonsunny nonexpansive retractions}

\author{
Heinz H.\ Bauschke\thanks{
Mathematics, University
of British Columbia,
Kelowna, B.C.\ V1V~1V7, Canada. E-mail:
\texttt{heinz.bauschke@ubc.ca}.},~~
Levi Miller\thanks{E-mail: \texttt{levi@levimiller.ca}.},
~~and~ Walaa M.\ Moursi\thanks{
Electrical Engineering,
Stanford 
University, 
Stanford, CA 94305, USA
and 
Mansoura University, Faculty of Science, Mathematics Department, 
Mansoura 35516, Egypt. E-mail:
\texttt{wmoursi@stanford.edu}.
}}

\date{May 22, 2018}

\maketitle

\centerline{\large \emph{Dedicated to Simeon Reich on the
occasion of his 70th
Birthday}}

\begin{abstract}
\noindent
Monotone operator theory and fixed point theory for nonexpansive
mappings are central areas in modern nonlinear analysis and
optimization. Although these areas are fairly well developed, 
almost all examples published are based on 
subdifferential operators, linear relations, or combinations
thereof. 

In this paper, we construct an intriguing maximally monotone
operator induced by a certain nonexpansive retraction.
We analyze this operator, which does not appear to be assembled
from subdifferential operators or linear relations, in some detail. 
Particular
emphasis is placed on duality and strong monotonicity. 
\end{abstract}
{ 

}

\section{Introduction}

Suppose that 
\begin{empheq}[box=\mybluebox]{equation}
\text{$X$ is a real Hilbert space,}
\end{empheq}
with inner product $\scal{\cdot}{\cdot}$ and induced norm
$\|\cdot\|$. 
Throughout, we assume that 
$X\neq \{0\}$ and that 
\begin{empheq}[box=\mybluebox]{equation}
e\in X \text{~and $\|e\|\leq 1$.}
\end{empheq}
Now define 
\begin{empheq}[box=\mybluebox]{equation}
\label{eq:def:Re}
R_e \colon X\to X\colon x\mapsto \|x\|\cdot e.
\end{empheq}

The mapping $R_e$ is nonexpansive and thus induces
an associated firmly nonexpansive mapping $T_e$ as well as
a maximally monotone operator $A_e$.

\emph{The goal of this paper is to present fundamental properties
of $R_e$, $T_e$, and $A_e$. These operators are neither
subdifferential operators nor linear relations; consequently, they provide
a new testing ground for properties in monotone operator
theory and fixed point theory.}

The paper is organized as follows.
In Section~\ref{sec:ReTeAe}, we focus on $R_e$, the induced
firmly nonexpansive mapping $T_e$ and the maximally monotone
operator $A_e$. 
Duality and stronger notions of monotonicity are considered in
Section~\ref{sec:Du}. 
We conclude the paper in Section~\ref{sec:last} with a discussion
on the resolvent iteration and the resolvent average. 

The notation employed is standard and follows, e.g., \cite{BC2017}. 
Finally, we assume the reader is familiar with basic monotone operator
theory and fixed point theory, as can be found in e.g., 
\cite{BC2017}, 
\cite{Brezis}, 
\cite{BurIus},
\cite{GK},
\cite{GR}, 
\cite{Rock98}, 
\cite{Simons2},
\cite{Simons1},
\cite{Zalinescu},
\cite{Zeidler1},
\cite{Zeidler2a},
or 
\cite{Zeidler2b}.

\section{$R_e$, $T_e$, and $A_e$}
\label{sec:ReTeAe}

We start by collecting some properties of $R_e$.
Item~\ref{p:retraction:i} of the following result states that
$\RP\cdot e$ is a nonexpansive retract of $X$,
via $R_e$. (See \cite{KR} for more on nonexpansive retracts.)

\begin{proposition}
\ 
\label{p:retraction}
\begin{enumerate}
\item 
\label{p:retraction:i}
If $\|e\|=1$, then $R_e$ is a nonexpansive
retraction of $\RP\cdot e$ and $R_e$ is not a Banach contraction. 
\item 
\label{p:retraction:iii}
If $\|e\|<1$, then $R_e$ is a Banach
contraction with optimal Lipschitz constant $\|e\|$ and $\Fix R_e=\{0\}$. 
\item 
\label{p:retraction:ii}
If $0<\|e\|\leq 1$ 
and $f\in\{e\}^\perp$ with $\|f\|=1$, then
$R_e$ is not sunny. 
\item 
\label{p:retraction:iv}
$R_e$ is nonexpansive. 
\end{enumerate}
\end{proposition}
\begin{proof}
We have 
\begin{equation}
\label{e:0509a}
(\forall x\in X)(\forall y\in X)\quad 
\|R_ex-R_ey\| = \big\| \|x\|e-\|y\|e\big\|
=\big|\|x\|-\|y\|\big|\cdot\|e\|
\leq \|e\|\cdot \|x-y\|,
\end{equation}
which shows that $R_e$ is Lipschitz continuous with constant
$\|e\|$; moreover, 
\begin{equation}
\label{e:0509aa}
(\forall x\in \RP\cdot e)(\forall y\in \RP\cdot e)\quad 
\|R_ex-R_ey\| = \|e\|\cdot \|x-y\|.
\end{equation}
It is clear that
$\Fix R_e \subseteq \ran R_e = \RP e$. 
If $x\in\RP\cdot e$, say $x=\rho e$, where $\rho\in\RP$,
then $R_e x =\|x\|e=\|\rho e\|e = \rho\|e\|e = \|e\|x$.
Hence 
\begin{equation}
\label{e:0509b}
(\forall x\in X)\quad x=R_e x \;\Leftrightarrow\; [x=0 \text{~or~} \|e\|=1].
\end{equation}

\ref{p:retraction:i}:
In view of \eqref{e:0509a} and the assumption that $\|e\|=1$,
it is clear that $R_e$ is nonexpansive.
From \eqref{e:0509b} we deduce that 
$\Fix R_e = \RP\cdot e$; thus, $R_e$ is a nonexpansive retract of
$\RP\cdot e$. 

\ref{p:retraction:iii}:
Combine \eqref{e:0509a} with \eqref{e:0509b}. 

\ref{p:retraction:ii}:
Set $x=\|e\|^{-1}e+\sqrt{3} f$. 
Then $\|x\|^2=1+3=4$ and so $R_e x =\|x\|e=2e$. 
Now consider $y=(x+R_e x)/2 =
2^{-1}(2+\|e\|^{-1})e+2^{-1}\sqrt{3}f\in [x,R_e x ]$. 
Then 
\begin{equation}
\|y\|^2 = \big(2^{-1}(2+\|e\|^{-1})\big)^2\|e\|^2 +
\big(2^{-1}\sqrt{3}\big)^2 
= 1+\|e\|+\|e\|^2 \leq 3 < 4 = \|x\|^2.
\end{equation}
We deduce that $R_e x=\|x\|e \neq \|y\|e = R_e y =\|y\|e$ 
and thus $R_e$ is not sunny. 

\ref{p:retraction:iv}:
Combine \ref{p:retraction:i} and \ref{p:retraction:iii}. 
\end{proof}

Proposition~\ref{p:retraction}\ref{p:retraction:iv} shows that
\begin{empheq}[box=\mybluebox]{equation}
\label{eq:def:Te}
T_e\colon X\to X\colon x\mapsto \tfrac{1}{2}x+\tfrac{1}{2}R_ex =
\tfrac{1}{2}x+\tfrac{1}{2}\|x\|e
\end{empheq}
is firmly nonexpansive and hence the resolvent $J_{A_e} =
(\Id+A_e)^{-1}$ of the maximally monotone operator
\begin{empheq}[box=\mybluebox]{equation}
\label{eq:def:Ae}
A_e = T_e^{-1}-\Id.
\end{empheq}

Our next task is to provide an explicit formula for $A_e$. 
It turns out that $A_e$ behaves quite differently, depending on
whether $\|e\|=1$ or $\|e\|<1$.

\begin{theorem}[{\bf $A_e$ for $\|e\|=1$}]
Suppose that $\|e\|=1$. Then
the maximally monotone operator $A_e$ is given by 
\begin{equation}
\label{e:0515a}
(\forall x\in X)\quad A_e x  = \begin{cases}
x-\frac{\|x\|^2}{\scal{e}{x}}\cdot e, &\text{if
$\scal{e}{x}>0$;}\\
\RM\cdot e, &\text{if $x=0$;}\\
\varnothing, &\text{otherwise.}
\end{cases}
\end{equation}
\end{theorem}
\begin{proof}
Denote the right-hand side of \eqref{e:0515a} by $B$, i.e.,
set 
\begin{equation}
(\forall x\in X)\quad B x  = \begin{cases}
x-\frac{\|x\|^2}{\scal{e}{x}}\cdot e, &\text{if
$\scal{e}{x}>0$;}\\
\RM\cdot e, &\text{if $x=0$;}\\
\varnothing, &\text{otherwise.}
\end{cases}
\end{equation}
Our job is to show that $B=A_e$, and for that it suffices to show
that $\gra B = \gra A_e = \menge{(T_ex,x-T_ex)}{x\in X}$ by the Minty
parametrization \cite{Minty}. 
For convenience, we also write $T$ instead of $T_e$. 

First, let $x\in X$. We need to consider two cases.

\emph{Case~1:} $\scal{e}{x}>0$.\\
Set $u = x-\|x\|^{2}\scal{e}{x}^{-1}e\in Bx$ and 
$y = x+u = 2x-\|x\|^{2}\scal{e}{x}^{-1}e$. 
Then 
\begin{align}\|u\|^2=\|x\|^2+\|x\|^4\scal{e}{x}^{-2}\|e\|^2
-2\|x\|^2\scal{e}{x}^{-1}\scal{x}{e}=
\|x\|^4\scal{e}{x}^{-2}-\|x\|^2
\end{align}
and 
\begin{equation}
\scal{x}{u} = \scal{x}{x-\|x\|^{2}\scal{e}{x}^{-1}e} = 
\scal{x}{x}-\|x\|^2\scal{e}{x}^{-1}\scal{x}{e}=0.
\end{equation}
Hence
$\|y\|^2 = \|x\|^2+\|u\|^2+2\scal{x}{u}=
\|x\|^2+(\|x\|^4\scal{e}{x}^{-2}-\|x\|^2) +2\cdot 0 = 
\|x\|^{4}\scal{e}{x}^{-2}$ and 
thus $\|y\|=\|x\|^2\scal{e}{x}^{-1}$.
Hence 
\begin{subequations}
\begin{align}
Ty &= \tfrac{1}{2}y+\tfrac{1}{2}\|y\|e = 
\tfrac{1}{2}(x+u)+\tfrac{1}{2}\|x\|^2\scal{e}{x}^{-1}e\\
&=\tfrac{1}{2}
\big(x+(x-\|x\|^2\scal{e}{x}^{-1}e)+\|x\|^2\scal{e}{x}^{-1}e\big)\\
&= x
\end{align}
\end{subequations}
and
\begin{equation}
y-Ty = (x+u)-x = u. 
\end{equation}
We have shown that $(x,u)=(Ty,y-Ty)\in \gra A_e$ as required. 

\emph{Case~2:} $x=0$.\\
Here we set $u=\eta e$, where $\eta \leq 0$, and
again $y=x+u=\eta e$. 
Then
$\|y\|=\|u\| = \|\eta e\|=|\eta|\cdot\|e\|=|\eta|=-\eta$.
Hence
$Ty =\tfrac{1}{2}y+\tfrac{1}{2}\|y\|e
=\tfrac{1}{2}\eta e + \tfrac{1}{2}(-\eta)e = 0 = x$
and
$y-Ty=\eta e - x= u-0=u$.
Again, we have shown that $(x,u)=(Ty,y-Ty)\in\gra A_e$, as
claimed. 

Combining \emph{Case~1} and \emph{Case~2}, we obtain the
conclusion
\begin{equation}
\label{e:0515b}
\gra B \subseteq \gra A_e.
\end{equation}

Conversely, let $y\in X$. 
Set $x = Ty = \tfrac{1}{2}y+\tfrac{1}{2}\|y\|e$
and note that $y-Ty = \tfrac{1}{2}y-\tfrac{1}{2}\|y\|e$.
Furthermore, 
$2\scal{e}{x}
=\scal{y}{e}+\|y\|\scal{e}{e}
=\scal{y}{e}+\|y\|
\geq -\|y\|\cdot\|e\|+\|y\|=0$
with equality if and only if $y\in\RM\cdot e$,
i.e., $x=0$, 
by Cauchy--Schwarz. 
Moreover,
$\|x\|^2 = 
\|\tfrac{1}{2}y+\tfrac{1}{2}\|y\|e\|^2
=
\tfrac{1}{4}\|y\|^2 + \tfrac{1}{4}\|y\|^2
+\tfrac{1}{2}\|y\|\scal{y}{e}
=\tfrac{1}{2}(\|y\|^2+\|y\|\scal{y}{e})
$
and
$\scal{e}{x}=\scal{e}{\tfrac{1}{2}y+\tfrac{1}{2}\|y\|e}
=\tfrac{1}{2}\scal{e}{y}+\tfrac{1}{2}\|y\|
$.

We now consider two conceivable alternatives.

\emph{Case~1}: $\scal{e}{x}>0$.\\
Then 
\begin{subequations}
\begin{align}
Bx &= x - \frac{\|x\|^2}{\scal{e}{x}}\cdot e
= \tfrac{1}{2}y+\tfrac{1}{2}\|y\|e - 
\frac{\tfrac{1}{2}(\|y\|^2+\|y\|\scal{y}{e})}{\tfrac{1}{2}\scal{e}{y}+\tfrac{1}{2}\|y\|}\cdot
e\\
&=\tfrac{1}{2}y-\tfrac{1}{2}\|y\|e \\
&=y-Ty. 
\end{align}
\end{subequations}
Therefore, $(Ty,y-Ty)=(x,Bx)\in\gra B$. 

\emph{Case~2}: $x=0$.\\
Then $y=\eta e$, where $\eta\in\RM$. 
Thus
$y-Ty=\tfrac{1}{2}y-\tfrac{1}{2}\|y\|e
=\tfrac{1}{2}\eta e - \tfrac{1}{2}|\eta|e
=\eta e \in B0=Bx
$.
Therefore, 
$(y,Ty)\in \{0\}\times B0 \subseteq \gra B$. 

Combining \emph{Case~1} with \emph{Case~2}, we deduce that
\begin{equation}
\label{e:0515c}
\gra A_e \subseteq \gra B.
\end{equation}
Finally, \eqref{e:0515b} and \eqref{e:0515c} yield the result.
\end{proof}

\begin{theorem}[{\bf $A_e$ for $\|e\|<1$}]
Suppose that $\|e\|<1$. 
Then the maximally monotone operator $A_e$ is given by 
\begin{equation}
\label{e:0515d}
(\forall x\in X)\quad A_e x = 
x +
\frac{2\scal{e}{x}-2\sqrt{(1-\|e\|^2)\|x\|^2+\scal{e}{x}^2}}{1-\|e\|^2}\cdot
e.
\end{equation}
\end{theorem}
\begin{proof}
Denote the right-hand side of \eqref{e:0515d} by $B$, i.e.,
set 
\begin{equation}
(\forall x\in X)\quad Bx  = 
x +
\frac{2\scal{e}{x}-2\sqrt{(1-\|e\|^2)\|x\|^2+\scal{e}{x}^2}}{1-\|e\|^2}\cdot
e.
\end{equation}
We also write
\begin{equation}
B x =x+\rho(x)\cdot e
\end{equation}
and observe that
$\rho(x)$ is the \emph{nonpositive} root of the quadratic equation
\begin{equation}
\label{e:quad}
(1-\|e\|^2)\rho^2 -4\scal{e}{x}\rho-4\|x\|^2=0.
\end{equation}
Once again, our job is to show that $B=A_e$, and for that it 
suffices to show
that $\gra B = \gra A_e = \menge{(T_ex,x-T_ex)}{x\in X}$ by the Minty
parametrization \cite{Minty}. 
For convenience, we also abbreviate $T=T_e$ and 
$\rho=\rho(x)$.

First, let $x\in X$ and set 
$y=x+Bx=2x+\rho e$. 
Then $\|y\|^2 = \|2x+\rho e\|^2
=4\|x\|^2+4\rho\scal{x}{e}+\rho^2\|e\|^2
=4\|x\|^2+4\rho\scal{x}{e}+\rho^2(\|e\|^2-1)+\rho^2 =
0+\rho^2=\rho^2$
by \eqref{e:quad};
thus, $\|y\|=|\rho| = -\rho$. 
Hence
\begin{equation}
Ty = \tfrac{1}{2}y+\tfrac{1}{2}\|y\|e
=\tfrac{1}{2}(2x+\rho e)+\tfrac{1}{2}(-\rho)e
=x
\end{equation}
and 
so 
\begin{equation}
y-Ty=(2x+\rho e)-x = x+\rho e = Bx.
\end{equation}
It follows that $(x,Bx)=(Ty,y-Ty)\in\gra A_e$ and thus
\begin{equation}
\label{e:0515e}
\gra B \subseteq \gra A_e.
\end{equation}

Conversely, let $y\in X$,
set $x = Ty = \tfrac{1}{2}y+\tfrac{1}{2}\|y\|e$ and note
that $y-Ty = \tfrac{1}{2}y-\tfrac{1}{2}\|y\|e$.
We have
$\|x\|^2 = \| \tfrac{1}{2}y+\tfrac{1}{2}\|y\|e\|^2
= \tfrac{1}{4}\|y\|^2+\tfrac{1}{4}\|y\|^2\|e\|^2
+\tfrac{1}{2}\|y\|\scal{y}{e}$
and 
$\scal{e}{x}=\scal{e}{ \tfrac{1}{2}y+\tfrac{1}{2}\|y\|e}
=\tfrac{1}{2}\scal{e}{y}+\tfrac{1}{2}\|y\|\|e\|^2$.
Hence
\begin{subequations}
\begin{align}
&\negthinspace\negthinspace\negthinspace 
(1-\|e\|^2)(-\|y\|)^2 - 4\scal{e}{x}(-\|y\|)-4\|x\|^2\\
&=(1-\|e\|^2)\|y\|^2+2\|y\|\scal{e}{y+\|y\|e} -
\big(\|y\|^2+\|y\|^2\|e\|^2+2\|y\|\scal{y}{e}\big)\\
&= 0.
\end{align}
\end{subequations}
In view of \eqref{e:quad} and the fact that $-\|y\|\leq 0$, it
follows that $-\|y\|=\rho(x)$. 
Hence
\begin{equation}
Bx=x+\rho(x)e 
=\big(\tfrac{1}{2}y+\tfrac{1}{2}\|y\|e\big)
-\|y\|e
=\tfrac{1}{2}y-\tfrac{1}{2}\|y\|e=y-Ty.
\end{equation}
Hence $(Ty,y-Ty)=(x,Bx)\in\gra B$ and thus
\begin{equation}
\label{e:0515f}
\gra A_e \subseteq \gra B.
\end{equation}

The conclusion now follows by combining
\eqref{e:0515e} with \eqref{e:0515f}. 
\end{proof}

\section{Duality and stronger notions of monotonicity}

\label{sec:Du}

If $A\colon X\To X$ is maximally monotone, 
then its \emph{dual} operator is $A^{-1}$, and the corresponding
dual objects 
of the resolvent and reflected resolvent are 
$\Id-J_A$ and $-R_A$, respectively \cite{BMW2012}. 
We now identify the dual objects, which have
pleasant explicit formulae, as well as some other interesting
properties.

\begin{lem}
\label{lem:micel:prop}
Recall that $A_e$, $T_e$, and $R_e$ are defined in
\eqref{eq:def:Ae}, \eqref{eq:def:Te}, and \eqref{eq:def:Re},
respectively. 
Then the following hold:
\begin{enumerate}
\item
{\rm \textbf{(duality)}}
\label{lem:micel:prop:a}
$-R_e = R_{-e}$,
$\Id-T_e = T_{-e}$, and
$(A_e)^{-1}=A_{-e}$.
\item
{\rm \textbf{(cone)}}
\label{lem:micel:prop:b}
$(\forall x\in X)(\forall\lambda\in\RPP)$
$A_e(\lambda x)=\lambda A_ex$. Consequently, $\gra A$ is a
(nonconvex) cone. 
\item
\label{lem:micel:prop:c}
If $\norm{e}=1$, then $(\forall (x,u)\in \gra A_e)$ $\innp{x,u}=0$.
\item
\label{lem:micel:prop:d}
If $\norm{e}=1$, then $\zer A_e = (A_e)^{-1}(0) =\RR_{+}\cdot e$.
\end{enumerate}
\end{lem}
\begin{proof}
\ref{lem:micel:prop:a}:
It is clear that $R_{(A_e)^{-1}}=-R_{A_e}=-R_e
=-\norm{\cdot}e=R_{-e}$.
Hence, $(A_e)^{-1}=A_{-e}$.
Consequently,
by the inverse resolvent identity 
and \eqref{eq:def:Te}, we
have
$\Id-T_e=\Id-J_{A_e}=J_{(A_e)^{-1}}
=J_{A_{-e}}=T_{-e}$. 
\ref{lem:micel:prop:b}:
This can be directly verified 
using \eqref{e:0515a}
and \eqref{e:0515d}.
\ref{lem:micel:prop:c}:
Let $x\in \dom A$.
If $x=0$ then
$(\forall u\in Ax=\RR_{-}\cdot e)$
we have $\innp{x,u}=0$.
Now suppose that $\innp{e,x}>0$.  
Then
$Ax$
is a singleton
 and $\innp{x,Ax}=\iNNp{x,x-\tfrac{\normsq{x}}{\innp{e,x}}\cdot e}
 =\normsq{x}-\normsq{x}=0$.
 \ref{lem:micel:prop:d}:
 Indeed, $\zer A_e=A_e^{-1}(0)=A_{-e}(0)
 =\RR_{-}\cdot(-e)=\RR_{+}\cdot e$.
\end{proof}

Strong monotonicity, paramonotonicity, cocoercivity, and $3^*$
monotonicity are perhaps the most important properties a
maximally monotone operators can have. 
The following two results provide a complete characterization of
these properties. Once again, the norm $\|e\|$ plays a crucial
role.

\begin{proposition}
\label{p:0521a}
Suppose that $\|e\|=1$ and that $X$ is not one-dimensional. 
Then the following hold:
\begin{enumerate}
\item
\label{p:0521ai}
$A_e$ is \emph{not} paramonotone.
Consequently, $A_e$ is neither strictly nor strongly monotone.
\item
\label{p:0521aii}
$A_e$ is  \emph{not} $3^*$ monotone.
\end{enumerate}
\end{proposition}
\begin{proof}
\ref{p:0521ai}:
Let $\alpha\in \left]1,+\infty\right[$,
let $x\in \dom A\smallsetminus{\RR_{+}\cdot e}$ 
and set $y=\alpha x$.
Note that $\innp{e,x}>0$,
that $Ay=A(\alpha x)=\alpha Ax\neq 0$ 
 by 
 Lemma~\ref{lem:micel:prop}\ref{lem:micel:prop:b}\&\ref{lem:micel:prop:d},
  and that $Ax$ (and consequently $Ay$) is a singleton.
Therefore,
by Lemma~\ref{lem:micel:prop}\ref{lem:micel:prop:c} 
$\innp{x-y,Ax-Ay}=\innp{(1-\alpha)x,(1-\alpha)Ax}
 =(1-\alpha)^2\innp{x,Ax}=0$.
However, $(x,Ay)\not\in \gra A$, hence $A$
is not paramonotone.

\ref{p:0521aii}:
Indeed, let $\alpha>0$,
let $f \in \{e\}^\perp$
such that $\norm{f}=1$,
set $x=2(f-e)$,
set $y=0$
and set $z=2\alpha f$.
We also write $T$ instead of $T_e$ for convenience.
Now
\begin{subequations}
\begin{align}
\innp{Tx-Tz, (\Id-T)y-(\Id-T)z}
&=-\innp{Tx-Tz, (\Id-T)z}\\
&=-\innp{f-e+\norm{f-e}e-\alpha(f+e),
\alpha(f-e)}\\
&=-\innp{f-e+\norm{f-e}e,\alpha(f-e)}
+\alpha^2(\normsq{f}-\normsq{e})\\
&=-\alpha\innp{f-e+\norm{f-e}e,f-e}\\
&=-\alpha(\normsq{f-e}-\norm{f-e})=-\alpha(2-\sqrt{2}).
\end{align}
\end{subequations}
Therefore, we conclude that
$\inf_{z\in X } \innp{Tx-Tz, (\Id-T)y-(\Id-T)z}=-\infty$,
hence $A$ is \emph{not}
$3^*$ monotone by \cite[Theorem~2.1(xvii)]{BMW2012}.
\end{proof}

\begin{prop}
\label{prop:ston:mono}
Suppose that $\norm{e}<1$. Then
the following hold:
\begin{enumerate}
\item
\label{prop:ston:mono:i}
$A_e$ is strongly monotone, with sharp constant $\tfrac{1-\norm{e}}{1+\norm{e}}$.
\item
\label{prop:ston:mono:iii}
$A_e$ is cocoercive, with sharp constant 
$\tfrac{1-\norm{e}}{1+\norm{e}}$. 
\item
\label{prop:ston:mono:ii}
$A_e$ is paramonotone.
\item
\label{prop:ston:mono:iv}
$A_e$ is $3^*$ monotone.
\item 
\label{prop:ston:mono:v}
$A_e$ is a displacement map if and 
only if $\norm{e}\le\tfrac{1}{3}$.
\end{enumerate}
\end{prop}
\begin{proof}
\ref{prop:ston:mono:i}:
Let $(x,y)\in X\times X$.
By the Minty parametrization \cite{Minty} of $\gra A_e$
$(\exists (u,v)\in X\times X)$
such that 
$(x,Ax)=(T_eu, u-T_eu)=\tfrac{1}{2}(u+\norm{u}e,u-\norm{u}e)$
 and 
$(y,Ay)=(T_ev, v-T_ev)=\tfrac{1}{2}(v+\norm{v}e,v-\norm{v}e)$.
Now, on the one hand, by the reverse triangle inequality we have 
\begin{subequations}
\label{eq:stmono:1}
\begin{align}
4\innp{x-y, Ax-Ay}&=\innp{(u-v)+(\norm{u}e-\norm{v}e),(u-v)-(\norm{u}e-\norm{v}e)} 
\\
&=\normsq{u-v}-(\norm{u}-\norm{v})^2\normsq{e}\ge (1-\normsq{e})\normsq{u-v}.
\end{align}
\end{subequations}
On the other hand, using the triangle inequality
and the reverse triangle inequality 
we have
\begin{equation}
\label{eq:stmono:2}
2\norm{x-y}=\norm{(u-v)+(\norm{u}-\norm{v})e}
\le \norm{u-v}+\big|\norm{u}-\norm{v}\big|\norm{e}\le (1+\norm{e})\norm{u-v}.
\end{equation}
Combining \eqref{eq:stmono:1}
 and \eqref{eq:stmono:2} yields
\begin{equation}
\innp{x-y, Ax-Ay}
\ge \tfrac{1-\normsq{e}}{(1+\norm{e})^2}\normsq{x-y}
=\tfrac{1-\norm{e}}{1+\norm{e}}\normsq{x-y}.
\end{equation}
To show that the strong monotonicity constant is sharp
we set $(x,y)= (e,0)$.
Now
\begin{subequations}
\begin{align}
\innp{e-0,Ae-A0}=\innp{e,Ae}
&=\iNNp{e, \Big(1+\tfrac{2\scal{e}{e}-2\sqrt{(1-\|e\|^2)\|e\|^2+\scal{e}{e}^2}}{1-\|e\|^2}\Big)e}\\
&=\Big(1+\tfrac{2\normsq{e}-2\norm{e}}{1-\normsq{e}}\Big)\normsq{e}
=\Big(\tfrac{\normsq{e}-2\norm{e}+1}{1-\normsq{e}}\Big)\normsq{e}\\
&=\tfrac{1-\norm{e}}{1+\norm{e}} \normsq{e-0}.
\end{align}
\end{subequations}
\ref{prop:ston:mono:iii}:
It follows from 
\ref{prop:ston:mono:i} applied with $e$
replaced by $-e$ that $A_{-e}$ is strongly monotone 
with sharp constant $\tfrac{1-\norm{e}}{1+\norm{e}}$.
Combining with 
\ref{lem:micel:prop:a}
and \cite[Example~22.7]{BC2017}
 we conclude that $A_e=(A_{-e})^{-1}$
 is $\tfrac{1-\norm{e}}{1+\norm{e}}$-cocoercive.
 \ref{prop:ston:mono:ii}:
Combine \ref{prop:ston:mono:iii}
and \cite[Example~22.8]{BC2017}.
 \ref{prop:ston:mono:iv}:
 Combine \ref{prop:ston:mono:iii}
  and
\cite[Example~25.20(i)]{BC2017}.
 \ref{prop:ston:mono:v}:
It follows from \cite[Proposition~4.13]{BBMW16}
 and  \ref{prop:ston:mono:iii}
in view of  \cite[Example~22.7]{BC2017}
that $A_e$ is a displacement map 
if and only if $\tfrac{1-\norm{e}}{1+\norm{e}}\ge \tfrac{1}{2}$;
equivalently $\norm{e}\le \tfrac{1}{3}$. 
\end{proof}

\begin{remark}[{\bf the one-dimensional case for $\|e\|=1$}]
Now suppose that $X$ is one-dimensional.
Then $A_e$ is a subdifferential operator; in fact,
$A_e=N_{\RP\, e}$.
Hence $A_e$ is both paramonotone (see, e.g.,
\cite[Example~22.4(i)]{BC2017}) and $3^*$ monotone (see, e.g.,
\cite[Example~25.14]{BC2017}), but neither
strictly nor strongly monotone.
\end{remark}

\begin{remark}[{\bf cyclic monotonicity}]
If $X$ is one-dimensional or $e=0$, 
then $A_e$ is clearly
cyclically monotone; otherwise, $A_e$ is not a gradient
by the formulae provided in Section~\ref{sec:ReTeAe}. 
Based on some numerical experiments, 
we conjecture that the degree of cyclic
monotonicity decreases from $+\infty$ (when $\|e\|=0$) to 
2 (when $\|e\|=1$) as $\|e\|$ increases; 
however, we do not know at which values the
transitions occur.
\end{remark}

\section{Resolvent iteration and average}

\label{sec:last}

It is well known 
that the sequence 
$((T_e)^nx_0)_\nnn$ converges weakly to some point in $\Fix
T_e = \zer A_e$; however, more can be said in our setting:

\begin{proposition}[{\bf resolvent iteration}]
\label{p:0521b}
Suppose that $X$ is not one-dimensional and let $x_0\in X$. 
The following hold true:
\begin{enumerate}
\item
\label{p:0521bi}
$((T_e)^nx_0)_\nnn$ converges \emph{strongly} to some point in $\Fix
T_e = \zer A_e$.
\item
\label{p:0521bii}
If $\|e\|<1$, then $(T_e)^nx_0\to 0$ linearly, with rate
$(1+\|e\|)/2$. 
\item
\label{p:0521biii}
If $\|e\|=1$, then $(T_e)^nx_0\to \|x_0\|\sinc(\theta_0)e$, where
$\sinc$ is the (unnormalized) sinc function and
$\theta_0\in[0,\pi]$ satisfies
$\cos(\theta_0)\|x_0\|=\scal{x_0}{e}$.
\end{enumerate}
\end{proposition}
\begin{proof}
\ref{p:0521bi}:
On the one hand, the convergence is known to be weak,
see, e.g., \cite[Proposition~5.16(iii)]{BC2017}. 
On the other hand, all iterates lie in $\spn\{e,x_0\}$.
Altogether, the convergence must be strong. 

\ref{p:0521bii}\&\ref{p:0521biii}:
Write $x_n=(T_e)^nx_0$ for every $\nnn$. 
The result is clear when $e=0$ since then $T_e =
\tfrac{1}{2}\Id$. Assume that $e\neq 0$ and that $x_0\neq 0$, 
and let us
work in ``polar coordinates'', i.e., 
pick $f\in \spn\{x_0,e\}\cap\{e\}^\perp$ such that $\|f\|=1$ and
write $x_n = \|x_n\|(\cos(\theta_n)\widehat{e}+\sin(\theta_n)f)$, where
$\widehat{e} = e/\|e\|$ and $\theta_n\in[0,\pi]$. 
Then 
\begin{equation}
\label{e:0522a}
\|x_{n+1}\|^2/\|x_n\|^2 =
(1+2\cos(\theta_n)\|e\|+\|e\|^2)/4
\end{equation}
and $\theta_n\to 0^+$. 
Thus $\|x_{n+1}\|^2/\|x_n\|^2\to (1+\|e\|^2 + 2\|e\|)/4
=((1+\|e\|)/2)^2$ and the result follows when $\|e\|<1$.

Now assume that $\|e\|=1$.
Then, by \eqref{e:0522a}, 
$\|x_{n+1}\|^2=\|x_n\|^2(1+\cos(\theta_n))/2 =
\|x_n\|^2\cos^2(\theta_n/2)$.
Inductively, it follows that 
\begin{subequations}
\begin{align}
x_n &= \|x_n\|\big(\cos(\theta_0/2^n)e+\sin(\theta_0/2^n)f\big)\\
&=\|x_0\|\Big(\prod_{k=1}^{n-1}\cos(\theta_0/2^k)\Big)\big(\cos(\theta_0/2^n)e+\sin(\theta/2^n)f\big)\\
&\to \|x_0\|\Big(\prod_{k=1}^\infty\cos(\theta_0/2^k)\Big)e\\
&=\|x_0\|\sinc(\theta_0)e \label{e:bla},
\end{align}
\end{subequations}
where we use \cite[equation~(1.3) on page~2]{Kac}
in \eqref{e:bla}. 
\end{proof}

We conclude this paper with 
an observation on the resolvent average that follows readily from
the definition. 

\begin{proposition}[{\bf resolvent average}]
Suppose that $\lambda_1,\lambda_2,\ldots,\lambda_m$ are in
$[0,1]$ such that $\sum_{i=1}^{m} {\lambda_i}=1$ and that
$e_1,\ldots,e_m$ are in the unit ball of $X$.
Set $\bar{e}=\sum_{i=1}^{m}\lambda_ie_i$.
Then $\sum_{i=1}^{m}\lambda_iR_{e_i} = R_{\bar{e}}$
and consequently the resolvent average \cite{BBMW16} of
$A_{e_1},\ldots,A_{e_m}$, with parameters
$\lambda_1,\ldots,\lambda_m$, is $A_{\bar{e}}$. 
\end{proposition}

\section*{Acknowledgments}
HHB was partially supported by the Natural Sciences and
Engineering Research Council of Canada.
WMM was partially supported by the Natural Sciences and
Engineering Research Council of Canada Postdoctoral Fellowship.

\end{document}